\documentclass[reqno,12pt]{amsart}
\newcommand{\RR}{{\mathbb R}}
\newcommand{\CC}{{\mathbb C}}

\newcommand{\NN}{{\mathbb N}}

\def\bege{\begin{equation}} \def\ende{\end{equation}}
\def\begr{\begin{eqnarray}} \def\endr{\end{eqnarray}}
\usepackage{amsmath}
\usepackage{amssymb}
\usepackage{cite}
\usepackage{bbm}
\usepackage{amsmath}
\usepackage{txfonts}
\usepackage{stmaryrd}
\usepackage{amssymb}
\usepackage{amsfonts}
\usepackage{times}
\usepackage{mathrsfs}
\usepackage{amsthm}
\usepackage{enumerate}
\usepackage{framed}
\usepackage{lipsum}
\usepackage{color}

\def\CC{ \mathbb{C}}
\newcommand{\DD}{{\mathbb D}}

\def\begr{\begin{eqnarray}} \def\endr{\end{eqnarray}}

\newtheorem{Lemma}{Lemma}
\newtheorem{Theorem}{Theorem}

\begin{document}
		\title[  ]{ Generalized weighted composition operators on weighted Hardy spaces}
	\author{ Lian Hu,  Songxiao Li$^\ast$ and  Rong Yang }
	\address{Lian Hu\\ Institute of Fundamental and Frontier Sciences, University of Electronic Science and Technology of China,
		610054, Chengdu, Sichuan, P.R. China.}
	\email{hl152808@163.com  }
	\address{Songxiao Li\\ Institute of Fundamental and Frontier Sciences, University of Electronic Science and Technology of China,
		610054, Chengdu, Sichuan, P.R. China. }
	\email{jyulsx@163.com}
	
	\address{Rong Yang\\ Institute of Fundamental and Frontier Sciences, University of Electronic Science and Technology of China,
		610054, Chengdu, Sichuan, P.R. China.}
	\email{yangrong071428@163.com }
	
	\begin{abstract}
	In this paper, we investigate the complex symmetric structure of generalized weighted composition operators $D_{m,\psi,\varphi}$ on the weighted Hardy space $H^2(\beta)$. We obtain explicit conditions for $ D_{m,\psi,\varphi}$ to be complex symmetric with the conjugation $J_w$. Under the assumption that $ D_{m,\psi,\varphi}$ is   $J_w$-symmetric, some sufficient and necessary conditions for  $D_{m,\psi,\varphi}$ to be Hermitian and normal are given.
	
	\thanks{$\ast$ Corresponding author.}
\end{abstract}
\maketitle

\section{Introduction}
   We denote by $\DD $ the open unit disc and by $H(\DD)$ the space of all analytic functions in $\DD$.  Let $\{\beta(n)\}$ be a sequence of positive number such that $\beta(0)=1$ and $\lim \inf\beta(n)^{1/n}\geq1$. The weighted Hardy space $H^2(\beta)$ consists of all   $f\in H(\DD)$  given by $f(z)=\sum_{n=0}^\infty a_nz^n$, such that$$
   \|f\|^2=\sum_{n=0}^\infty |a_n|^2\beta(n)^2<\infty.$$
   Every weighted Hardy space $H^2(\beta)$ is a Hilbert space. The weight sequence for $H^2(\beta)$ is written as $ \beta(n)=\|z^n\|$. The set $\{e_n(z)={z^n\over \beta(n)}\}_{n\ge 0}$ forms an orthonormal basis for the space $ H^2(\beta)$. For $ f(z)=\sum_{n=0}^\infty a_nz^n$ and $ g(z)=\sum_{n=0}^\infty c_nz^n$ in $H^2(\beta) $, the inner product on $H^2(\beta)$ is given by$$
     \langle f, g  \rangle=\sum_{n=0}^\infty a_n\overline{c_n}\beta(n)^2.$$
   $H^2(\beta)$ is a reproducing kernel Hilbert space of analytic functions which means that the point evaluations of functions on $H^2(\beta)$ are bounded linear functions. For any point $\alpha$ in $\DD$, define$$
  K_\alpha(z)=\sum_{n=0}^\infty{ \bar{\alpha}^n \over  \beta(n)^2}z^n,\,\,z\in\DD.$$
  Obviously, $K_\alpha$ is the reproducing kernel function for $H^2(\beta)$:$$
   \langle f,K_\alpha\rangle=f(\alpha)$$
   for any $f$ in $H^2(\beta)$. For each point $\alpha$ in $\DD$ and positive integer $m$, evaluation of the $m^{th}$ derivative of functions in $H^2(\beta)$ at $\alpha$ is a bounded linear functional and $f^{(m)}(\alpha)=\langle f, K_\alpha^{[m]}\rangle$ (see \cite{cm}), where $$
   K_\alpha^{[m]}(z)=\sum_{n=m}^\infty{n! \over (n-m)! }\bar{\alpha}^{n-m}{z^n  \over \beta(n)^2 }.$$
    Clearly, the   Hardy space $H^2$, the  Bergman space $A^2$, the Dirichlet space $\mathcal{D}$ and the derivative Hardy space $S^2 $   are the weighted Hardy spaces which are identified with the weighted sequences $\beta(n)=1$, $ \beta(n)=(n+1)^{-{1\over 2}}$, $ \beta(n)=n^{1\over2}$ and $\beta(n)=n$, respectively.

    Let $m\in \NN$, $\psi\in H(\DD)$ and $\varphi$ be an analytic self-map of $\DD$.      The generalized weighted composition operator  $D_{m,\psi,\varphi}$ (see \cite{zxl1, zxl2, zxl3}) is defined by
       $$
    D_{m,\psi,\varphi}f(z)=\psi(z)f^{(m)}(\varphi(z)), \,\,f\in H(\DD), \,z\in \DD.$$
    If $m=0$, the operator $D_{m,\psi,\varphi}$ becomes the weighted composition operator, which is always denoted by $\psi C_{\varphi}$. If   $\psi =1$ and $m=0$, the operator $D_{m,\psi,\varphi}$ is the composition operator $  C_{\varphi}$.  When $\psi =1$ and $m=1$, the operator $D_{m,\psi,\varphi}$ is called the composition-differentiation operator  and  denoted by $D_\varphi  $. When    $m=1$, the operator $D_{m,\psi,\varphi}$ is called the weighted composition-differentiation operator  and  denoted by $\psi D_\varphi  $.
    In \cite{fh}, Fatehi and Hammond obtained the adjoint, norm and spectrum of   $  D_\varphi$ on the Hardy space $H^2$. Some properties of weighted composition-differentiation operators  were investigated in \cite{fh2, hw, hw2, lpx}. See \cite{djo, mg, zhang, zxl1, zxl2, zxl3, zxl4, zxl5} for more results on generalized weighted composition operators on analytic function spaces.

   An operator $C$ is called  a conjugation on complex Hilbert space $\mathcal{H}$ if it satisfies the following conditions:
    \begin{enumerate}
    	\item[(i)] conjugate-linear or anti-linear: $C(a x +b y)=\bar{a}C(x)+\bar{b}C(y)$, for any $x,y\in \mathcal{H}$ and $a,b\in \mathbb{C}$;
    	\item[(ii)] isometric: $\|Cx\|=\|x\|$, for any $x\in \mathcal{H}$;
    	\item[(iii)] involutive: $C^2=I$, where $I$ is an identity operator.
    \end{enumerate}

   The operator $J$, defined as $Jf(z)=\overline{f(\bar{z})}$,  is a standard conjugation. In this paper, we consider a generalized conjugation $ J_w$, which is defined as follows: $$J_wf(z)=\overline{f(w\bar{z})},\,\,z\in \DD,$$
   where $f\in H^2(\beta)$ and $w\in \CC$ with $|w|=1$.

    A bounded linear operator $T$ is said to be complex symmetric (complex symmetric with $C$ or $C$-symmetric) if there is a conjugation $C$ on a Hilbert space $\mathcal{H}$ such that
   $$ T=CT^\ast C.$$  It follows from \cite{gpp} that operator $T$ is complex symmetric if and only if $T$ has a self-transpose matrix representation with respect to an orthonormal basis. Complex symmetric operators can be regarded as a generalization of complex symmetric matrices.  In \cite{gp,gp2,gw,gw2}, Garcia, Putinar and Wogen initiated the general study of complex symmetric operators. Examples of complex symmetric operators include  normal operators, binormal operators, Hermitian operators, compressed Toeplitz operators  and Hankel operators.   In recent decades, complex symmetric composition operators and weighted composition operators acting on some Hilbert spaces of analytic functions  have been  studied considerably.  See \cite{elm, f, fh, fh2, gm, gh, gp,gp2,gw,gw2, gz, hw, hw2, jhz, jkkl, lk, lpx, mg, mg2, ns, nst, tmh, y,zz}
   for more results   on complex symmetric operators.

    Garcia and Hammond in \cite{gh} gave several   classes of   $J$-symmetric composition operators and weighted composition operators on  $H^2(\beta)$. In \cite{mg2}, Malhotra and Gupta characterized  complex symmetric   weighted composition  operators on $H^2(\beta)$. Complex symmetric weighted composition-differentiation operators on the Hardy space $H^2$ were investigated by Han and Wang in  \cite{hw}. Complex symmetric  weighted composition-differentiation operators $\psi D_\varphi  $  on the weighted Bergman space $A^2_\alpha$ and the derivative Hardy space were characterized in \cite{lpx}.  In \cite{hw2}, Han and Wang studied complex symmetric generalized weighted composition  operators on the Bergman space $A^2$.

   In this paper, we investigate  the symbols  $\psi$ and   $\varphi$ give rise to   $J_w$-symmetric generalized weighted composition  operator  $D_{m,\psi,\varphi}$ on $H^2(\beta)$. As an application, we give some necessary and sufficient conditions for $J_w$-symmetric  operator  $D_{m,\psi,\varphi}$   to be Hermitian and normal. 

\section{Main results and proofs}
      In this section, we state and prove our main results in this paper.   For this purpose, we need the following lemma, which will be used in proving our main result.

\begin{Lemma}\label{lem1}
	Let $m\in \NN$, $ \varphi$ be an analytic self-map of $\DD$ and $\psi \in H(\DD)$ such that $D_{m,\psi,\varphi}$ is bounded on $H^2(\beta)$. Then for any $ \alpha \in \DD$,
	$$ D_{m,\psi,\varphi}^\ast K_\alpha(z) =\overline{\psi(\alpha)}K_{\varphi(\alpha)}^{[m]}(z), ~~~~~z\in \DD.$$
\end{Lemma}

\begin{proof}
	For any $f\in H^2(\beta)$, we have
	\begin{align*}
	\langle f, D_{m,\psi,\varphi}^\ast K_\alpha \rangle=&	\langle D_{m,\psi,\varphi}f, K_\alpha \rangle = \psi(\alpha)f^{(m)}(\varphi(\alpha))\\=& \psi(\alpha) \langle f, K_{\varphi(\alpha)}^{[m]} \rangle = \langle f, \overline{\psi(\alpha)}K_{\varphi(\alpha)}^{[m]} \rangle,
	\end{align*}
which implies the desired result. \end{proof}

The following theorem  gives the characterization  of  $\psi$ and   $\varphi$ such that the  operator $D_{m,\psi,\varphi}$ is $J_w $-symmetric  on $H^2(\beta)$.

\begin{Theorem}\label{the1}
	Let $m\in \NN$, $ \varphi$ be an analytic self-map of $\DD$ and $\psi\in H(\DD)$ be not identically zero such that $D_{m,\psi,\varphi}$ is bounded on $H^2(\beta)$.
If $D_{m,\psi,\varphi}$ is $J_w $-symmetric on $H^2(\beta)$, then
		\begin{align}\label{equ1}
	\varphi(z)=a_0+{\beta(m+1)^2a_1q(z)  \over (m+1)\bar{w}^{m+1}\beta(m)^2p(z) }
	\end{align}
and
	\begin{align}\label{equ2}
		\psi(z)={ \beta(m)^2a_2 \over (m!)^2 }K^{[m]}_{w\overline{a_0}}(z),
	\end{align}
	where $ a_0=\varphi(0)$, $ a_1=\varphi'(0)$, $a_2=\psi^{(m)}(0)$,
		\begin{align}\label{equ11}
	 p(z)=\sum_{n=m}^{\infty}{  n! \over (n-m)! }{ (\bar{w}a_0z)^{n-m} \over \beta(n)^2}
	 	\end{align}
 	and
 		\begin{align}\label{equ26}
 			 q(z)=\sum_{n=m+1}^{\infty}{  n! \over (n-m-1)! }{ \bar{w}^na_0^{n-m-1}z^{n-m} \over \beta(n)^2 }.
 				\end{align}
	
	 Conversely, let $a_0, a_1\in \DD$ and $a_2\in \CC$. If $\varphi$ and $\psi$ are analytic maps of $\DD$, defined as in equations (\ref{equ1}) and (\ref{equ2}), then $D_{m,\psi,\varphi}$ is   $J_w $-symmetric  on $H^2(\beta)$ only if $a_0=0$ or $a_1=0$ or both are $0$.
\end{Theorem}

\begin{proof}
	Assume that $D_{m,\psi,\varphi}$ is $J_w $-symmetric on $H^2(\beta)$. Then for any $z,\alpha\in \DD $,
		\begin{align}\label{equ299}
J_wD^\ast_{m,\psi,\varphi} K_\alpha(z)=	D_{m,\psi,\varphi}J_wK_\alpha(z).
\end{align}
	 Lemma \ref{lem1} yields that
	 \begin{align*}
	 	J_wD^\ast_{m,\psi,\varphi} K_\alpha(z)=&J_w\overline{\psi(\alpha)}K^{[m]}_{\varphi(\alpha)}(z)\\
	 	=&J_w\overline{\psi(\alpha)}\sum_{n=m}^\infty{n!\over (n-m)!}{ \overline{\varphi(\alpha)}^{n-m}z^n \over  \beta(n)^2 }\\
	 	=&\psi(\alpha)	 \sum_{n=m}^\infty{n!\over (n-m)!}{\varphi(\alpha)^{n-m}(\bar{w}z)^n \over  \beta(n)^2  }	
	 	\end{align*}
 	and
 	\begin{align*}
 	D_{m,\psi,\varphi}J_wK_\alpha(z)=&D_{m,\psi,\varphi}J_w\sum_{n=0}^\infty{ (\bar{\alpha}z)^n \over \beta(n)^2  }=D_{m,\psi,\varphi}\sum_{n=0}^\infty{ (\alpha\bar{w}z)^n \over \beta(n)^2 }\\
 	=&	D_{m,\psi,\varphi}K_{w\bar{\alpha}}(z)=\psi(z) K^{(m)}_{w\bar{\alpha}}(\varphi(z))\\
 	=&\psi(z)\sum_{n=m}^\infty{n!\over (n-m)!}{\varphi(z)^{n-m}(\bar{w}\alpha)^n \over  \beta(n)^2  }	
 \end{align*}	
for any $z,\alpha\in \DD $. Hence, equation (\ref{equ299}) is equivalent to
	\begin{align}\label{equ3}
\psi(\alpha)	 \sum_{n=m}^\infty{n!\over (n-m)!}{\varphi(\alpha)^{n-m}(\bar{w}z)^n \over  \beta(n)^2  }		=	\psi(z)\sum_{n=m}^\infty{n!\over (n-m)!}{\varphi(z)^{n-m}(\bar{w}\alpha)^n \over  \beta(n)^2  }
		\end{align}
 for any $z,\alpha\in \DD $. Let $ \alpha=0$ in (\ref{equ3}). We obtain that
    $$
	\psi(0) \sum_{n=m}^\infty{n!\over (n-m)!}{\varphi(0)^{n-m}(\bar{w}z)^n \over  \beta(n)^2  }	=0 $$
for any $z\in \DD $, which means that $ \psi(0)=0$.

Let $\psi(z)=z^kh(z)$, where $k$ is a positive integer and $h$ is analytic on $\DD$ with $h(0)\ne 0$. Next we claim that $k=m$. If $k>m$, equation (\ref{equ3}) is equivalent to
\begin{align*}
\alpha^{k-m}h(\alpha) \sum_{n=m}^\infty{n!\over (n-m)!}{\varphi(\alpha)^{n-m}\bar{w}^nz^{n-m}\over  \beta(n)^2  }=z^{k-m}h(z)\sum_{n=m}^\infty{n!\over (n-m)!}{\varphi(z)^{n-m}\bar{w}^n\alpha^{n-m} \over  \beta(n)^2  }
\end{align*}
for any $z,\alpha\in \DD $. Setting $\alpha=0$, we have that $h\equiv0$, which contradicts with $h(0)\ne 0$. If $ k<m$,  the equation (\ref{equ3}) is equivalent to
\begin{align*}
	z^{m-k}h(\alpha) \sum_{n=m}^\infty{n!\over (n-m)!}{\varphi(\alpha)^{n-m}\bar{w}^nz^{n-m}\over  \beta(n)^2  }=\alpha^{m-k}h(z)\sum_{n=m}^\infty{n!\over (n-m)!}{\varphi(z)^{n-m}\bar{w}^n\alpha^{n-m} \over  \beta(n)^2  }.
\end{align*}
 Setting $\alpha=0$, we have that $ h(0)=0$, which contradicts with $h(0)\ne 0$.
Thus $k=m$ and the equation (\ref{equ3}) becomes
\begin{align}\label{equ5}
	h(\alpha) \sum_{n=m}^\infty{n!\over (n-m)!}{\varphi(\alpha)^{n-m}\bar{w}^nz^{n-m}\over  \beta(n)^2  }=h(z)\sum_{n=m}^\infty{n!\over (n-m)!}{\varphi(z)^{n-m}\bar{w}^n\alpha^{n-m} \over  \beta(n)^2  }
\end{align}
for any $z,\alpha\in \DD $. Let $\alpha=0$ in (\ref{equ5}). We get
  $$
	h(0) \sum_{n=m}^\infty{n!\over (n-m)!}{\varphi(0)^{n-m}\bar{w}^nz^{n-m}\over  \beta(n)^2  }=h(z){ m!\bar{w}^m \over \beta{(m)}^2 },$$ that is
$$
 h(z)={h(0)\beta(m)^2  \over m! } \sum_{n=m}^\infty{n!\over (n-m)!}{\varphi(0)^{n-m}\bar{w}^{n-m}z^{n-m}\over  \beta(n)^2  }.$$
 Therefore,
 \begin{equation}\label{equ9}
 	\begin{aligned}
 \psi(z)=z^mh(z)=&{h(0)\beta(m)^2  \over m! } \sum_{n=m}^\infty{n!\over (n-m)!}{\varphi(0)^{n-m}\bar{w}^{n-m}z^{n}\over  \beta(n)^2  }\\
 =&{\psi^{(m)}(0)\beta(m)^2  \over (m!)^2 } \sum_{n=m}^\infty{n!\over (n-m)!}{\varphi(0)^{n-m}\bar{w}^{n-m}z^{n}\over  \beta(n)^2  }\\
 =&{\psi^{(m)}(0)\beta(m)^2  \over (m!)^2 }K^{[m]}_{w\overline{\varphi(0)}}(z),
	\end{aligned}
\end{equation}
where $\psi^{(m)}(0)=m!h(0)\ne0$. Substituting $\psi(z)$ in (\ref{equ3}), we obtain that
\begin{align}\label{equ6}
 K^{[m]}_{w\overline{\varphi(0)}}(\alpha)	 \sum_{n=m}^\infty{n!\over (n-m)!}{\varphi(\alpha)^{n-m}(\bar{w}z)^n \over  \beta(n)^2  }		= K^{[m]}_{w\overline{\varphi(0)}}(z)\sum_{n=m}^\infty{n!\over (n-m)!}{\varphi(z)^{n-m}(\bar{w}\alpha)^n \over  \beta(n)^2  }
 \end{align}
for any $z,\alpha \in \DD$.
Let $$F_1(z)=\sum_{n=m}^\infty{n!\over (n-m)!}{\varphi(\alpha)^{n-m}(\bar{w}z)^n \over  \beta(n)^2  }	
$$and $$
F_2(z)=\sum_{n=m}^\infty{n!\over (n-m)!}{\varphi(z)^{n-m}(\bar{w}\alpha)^n \over  \beta(n)^2  }.$$ It is clear that $N^{th}$ derivative of $ K^{[m]}_{w\overline{\varphi(0)}}$ is equal to $0$ at $z=0$, that is, $ K^{[m]}_{w\overline{\varphi(0)}}(0)=0$, where $N=1,2,\cdots,m-1$. In addition, we have$$	\left( K^{[m]}_{w\overline{\varphi(0)}}(z) \right)^{(m)}=\sum_{n=m}^\infty{(n!)^2\over [(n-m)!]^2}{\varphi(0)^{n-m}\bar{w}^{n-m}z^{n-m}\over  \beta(n)^2  }, $$
$$\left( K^{[m]}_{w\overline{\varphi(0)}}(z) \right)^{(m+1)}=\sum_{n=m+1}^\infty{(n!)^2\over (n-m)!(n-m-1)!}{\varphi(0)^{n-m}\bar{w}^{n-m}z^{n-m-1}\over  \beta(n)^2  },
$$
$$F_1^{(m+1)}=\sum_{n=m+1}^\infty{(n!)^2\over (n-m)!(n-m-1)!}{ \bar{w}^n\varphi(\alpha)^{n-m}z^{n-m-1}\over  \beta(n)^2   }
$$and $$
 F_2'(z)=\sum_{n=m}^\infty{n!\over (n-m-1)!}{\varphi(z)^{n-m-1}(\bar{w}\alpha)^n \varphi'(z)\over  \beta(n)^2  }.$$
Therefore,
differentiating the equation (\ref{equ6}) $(m+1)$ times with respect to $z$, we have
\begin{equation}\label{equ7}
	\begin{aligned}
	&	\sum_{i=0}^{m+1}{m+1\choose i}\left(  K^{[m]}_{w\overline{\varphi(0)}}(z) \right)^{(i)}F_2(z)^{(m+1-i)}\\
=&\sum_{n=m+1}^\infty{(n!)^2\over (n-m)!(n-m-1)!}{\varphi(0)^{n-m}\bar{w}^{n-m}z^{n-m-1}\over  \beta(n)^2  }\cdot\sum_{n=m}^\infty{n!\over (n-m)!}{\varphi(z)^{n-m}(\bar{w}\alpha)^n \over  \beta(n)^2  }\\+&\sum_{i=0}^{m-1}{m+1\choose i}\left(  K^{[m]}_{w\overline{\varphi(0)}}(z) \right)^{(i)}F_2^{(m+1-i)}(z)+(m+1)\sum_{n=m}^\infty{(n!)^2\over [(n-m)!]^2}{\varphi(0)^{n-m}\bar{w}^{n-m}z^{n-m}\over  \beta(n)^2  }\cdot\\&\,\,\,\,\,\,\,\,\,\,\,\,\,\,\,\,\,\,\,\,\,\,\,\,\,\,\,\,\,\,\,\,\,\sum_{n=m}^\infty{n!\over (n-m-1)!}{\varphi(z)^{n-m-1}(\bar{w}\alpha)^n \varphi'(z)\over  \beta(n)^2  }\\
	=&K^{[m]}_{w\overline{\varphi(0)}}(\alpha)\sum_{n=m+1}^\infty{(n!)^2\over (n-m)!(n-m-1)!}{ \bar{w}^n\varphi(\alpha)^{n-m}z^{n-m-1}\over  \beta(n)^2   }.
		\end{aligned}
	\end{equation}
Let $z=0$ in (\ref{equ7}). We get
\begin{align*}
		&{[(m+1)!]^2\bar{w}\varphi(0)\over \beta(m+1)^2}\sum_{n=m}^\infty{n!\over (n-m)!}{\varphi(0)^{n-m}(\bar{w}\alpha)^n \over  \beta(n)^2  }\\+&{(m+1)(m!)^2 \over \beta(m)^2 }\sum_{n=m}^\infty{n!\over (n-m-1)!}{\varphi(0)^{n-m-1}(\bar{w}\alpha)^n \varphi'(0)\over  \beta(n)^2  }\\
		=&{ [(m+1)!]^2\bar{w}^{m+1}\varphi(\alpha) \over \beta(m+1)^2 }K^{[m]}_{w\overline{\varphi(0)}}(\alpha)
	\end{align*}
for any $\alpha\in \DD$. Thus
\begin{equation}\label{equ10}
	\begin{aligned}
		&{[(m+1)!]^2\bar{w}\varphi(0)\over \beta(m+1)^2}\sum_{n=m}^\infty{n!\over (n-m)!}{\varphi(0)^{n-m}\bar{w}^n\alpha^{n-m} \over  \beta(n)^2  }\\+&{(m+1)(m!)^2 \over \beta(m)^2 }\sum_{n=m}^\infty{n!\over (n-m-1)!}{\varphi(0)^{n-m-1}\bar{w}^n\alpha^{n-m} \varphi'(0)\over  \beta(n)^2  }\\
	=&{ [(m+1)!]^2\bar{w}^{m+1}\varphi(\alpha) \over \beta(m+1)^2 }\sum_{n=m}^\infty{n!\over (n-m)!}{\varphi(0)^{n-m}\bar{w}^{n-m}\alpha^{n-m}\over  \beta(n)^2  }
	\end{aligned}
\end{equation}
for any $\alpha\in \DD$.
Hence, (\ref{equ10}) deduces that $$
	 \varphi(\alpha)=\varphi(0)+{\beta(m+1)^2\varphi'(0)q(\alpha)  \over (m+1)\bar{w}^{m+1}\beta(m)^2p(\alpha) }, $$	where $$
	  p(\alpha)=\sum_{n=m}^{\infty}{  n! \over (n-m)! }{ (\bar{w}a_0\alpha)^{n-m} \over \beta(n)^2}$$ and$$ q(\alpha)=\sum_{n=m+1}^{\infty}{  n! \over (n-m-1)! }{ \bar{w}^na_0^{n-m-1}\alpha^{n-m} \over \beta(n)^2 }. $$

Conversely,	let $a_0, a_1\in \DD$ and $a_2\in \CC$,
 $$	\varphi(z)=a_0+{\beta(m+1)^2a_1q(z)  \over (m+1)\bar{w}^{m+1}\beta(m)^2p(z) } ~and~  \psi(z)={ \beta(m)^2a_2 \over (m!)^2 }K^{[m]}_{w\overline{a_0}}(z),$$
where $p(z)$ and $q(z)$ are defined as (\ref{equ11}) and (\ref{equ26}). Then  for $J_w $-symmetric operator $D_{m,\psi,\varphi}$, equation (\ref{equ3}) must hold. This is equivalent to
\begin{equation}\label{equ12}
	\begin{aligned}	
		&\sum_{n=m}^\infty{n!\over (n-m)!}{a_0^{n-m}\bar{w}^{n-m}\alpha^{n}\over  \beta(n)^2  }\left(	\sum_{n=1}^\infty{ n!\bar{w}^nz^n \over (n-m)!\beta(n)^2 }\left(a_0+{\beta(m+1)^2a_1q(\alpha)  \over (m+1)\bar{w}^{m+1}\beta(m)^2p(\alpha) } \right)^{n-m}\right)\\
		=&\sum_{n=m}^\infty{n!\over (n-m)!}{a_0^{n-m}\bar{w}^{n-m}z^{n}\over  \beta(n)^2  }\left(	\sum_{n=1}^\infty{ n!\bar{w}^n\alpha^n \over (n-m)!\beta(n)^2 }\left(a_0+{\beta(m+1)^2a_1q(z)  \over (m+1)\bar{w}^{m+1}\beta(m)^2p(z) } \right)^{n-m}\right).
	\end{aligned}
	\end{equation}
For any $\alpha,z\in \DD$, ${ q(z) \over  p(z)}$ is analytic and $q(0)=0$. Thus ${ q(z) \over  p(z)}$ can be written as
\begin{align}\label{equ22}
{ q(z) \over  p(z)}=\sum_{i=1}^\infty c_i\bar{w}^{i+m}a_0^{i-1}z^{i},
\end{align}
where $c_1=1$ and $c_i\in \RR,i=2,3,\cdots,$ Therefore, (\ref{equ12}) is equivalent to
	\begin{align*}
	&\sum_{n=m}^\infty { n!\bar{w}^{n-m}a_0^{n-m}\alpha^n \over  (n-m)!\beta(n)^2 }	\sum_{l=m}^\infty{ l!\bar{w}^lz^l \over (l-m)!\beta(l)^2 }\\
	&\,\,\,\,\,\,\,\,\,\,\,\,\,\,\,\,\,\,\,\,\,\,\,\,\,\,\,\,\,\,\,\,\,\,\,\,\,\cdot\sum_{k=0}^{l-m}{ l-m \choose k }a_0^k\left( \sum_{i=1}^\infty{\beta(m+1)^2a_1 \over (m+1)\bar{w}^{m+1}\beta(m)^2 }c_i\bar{w}^{i+m}a_0^{i-1}\alpha^{i} \right)^{l-m-k}\\
	=&\sum_{n=m}^\infty { n!\bar{w}^{n-m}a_0^{n-m}z^n \over  (n-m)!\beta(n)^2 }	\sum_{l=m}^\infty{ l!\bar{w}^l\alpha^l \over (l-m)!\beta(l)^2 }\\
	&\,\,\,\,\,\,\,\,\,\,\,\,\,\,\,\,\,\,\,\,\,\,\,\,\,\,\,\,\,\,\,\,\,\,\,\,\,\cdot\sum_{k=0}^{l-m}{ l-m \choose k }a_0^k\left( \sum_{i=1}^\infty{\beta(m+1)^2a_1 \over (m+1)\bar{w}^{m+1}\beta(m)^2 }c_i\bar{w}^{i+m}a_0^{i-1}z^{i} \right)^{l-m-k}
		\end{align*}
	for any $\alpha,z\in \DD$.
	Considering the coefficient of $z^{m+2}\alpha^{m+1}$, we obtain that
	\begin{equation}\label{equ13}
		\begin{aligned}
			&{ (m+1)!\bar{w}^{m+1} \over \beta(m+1)^2 }\left( { m!\beta(m+1)^2 \over (m+1)\beta(m)^4 }c_2\bar{w}a_0a_1+{ m! \over\beta(m)^2 }c_1\bar{w}a_0a_1+{ (m+2)!\bar{w}^2  \over 2\beta(m+2)^2 }a_0^3 \right)\\
			=&{ (m+2)!\bar{w}^{m+2} \over  2\beta(m+2)^2 }\left( { (m+1)!\bar{w} \over \beta(m+1)^2 }a_0^3+{2m! \beta(m+1)^2 \over (m+1)\beta(m)^4 }c_1a_0a_1 \right).
\end{aligned}
\end{equation}	
	Therefore, equation (\ref{equ13}) holds only if $a_0=0$ or $a_1=0$ or both are zero. Next, we consider the following three cases:
	
	Case 1. $a_0=0$ and $a_1\ne 0$. In this case, $$\varphi(z)=a_1z ~{\rm and}~ \psi(z)={a_2\over m!}z^m.$$ Then
		\begin{align*}
		J_wD^\ast_{m,\psi,\varphi} K_\alpha(z)=&{a_2\over m!} \sum_{n=m}^\infty{ n! \over (n-m)!} { a_1^{n-m}(\bar{w}\alpha z)^n \over \beta{(n)}^2}\\
		=&	D_{m,\psi,\varphi}J_wK_\alpha(z).
		\end{align*}
	
	Case 2. $a_0\ne0$ and $a_1= 0$.	In this case, $$\varphi(z)=a_0  ~{\rm and}~ \psi(z)={ \beta(m)^2a_2 \over (m!)^2 }K^{[m]}_{w\overline{a_0}}(z).$$ Then
	\begin{align*}
		J_wD^\ast_{m,\psi,\varphi} K_\alpha(z)=&{\beta(m)^2 \bar{w}^ma_2 \over (m!)^2 } \sum_{n=m}^\infty{n!\over (n-m)!}{a_0^{n-m}\bar{w}^{n-m}z^{n}\over  \beta(n)^2  }\sum_{n=m}^\infty{n!\over (n-m)!}{a_0^{n-m}\bar{w}^{n-m}\alpha^{n}\over  \beta(n)^2  }\\
		=&	D_{m,\psi,\varphi}J_wK_\alpha(z).
		\end{align*}
	
	Case 3. $a_0=a_1=0$. In this case, $$\varphi(z)=0 ~{\rm and}~ \psi(z)={a_2\over m!}z^m.$$ Then $$
		J_wD^\ast_{m,\psi,\varphi} K_\alpha(z)= { a_2 \over  \beta(m)^2}(\bar{w}\alpha z)^m=	D_{m,\psi,\varphi}J_wK_\alpha(z).$$
	The proof is complete.
\end{proof}
	
	The following result obtains the condition on $\varphi$ so that $\varphi$ is an automorphism on $\DD$ and $D_{m,\psi,\varphi}$ is $J_w $-symmetric on $H^2(\beta)$.
	
	\begin{Theorem}
		Let $m\in \NN$, $\varphi$ be an automorphism on $\DD$   and $\psi\in H(\DD)$ be not identically zero such that $D_{m,\psi,\varphi}$ is $J_w $-symmetric  on $H^2(\beta)$. Then one of the following statements hold:
		\begin{enumerate}
			\item[(i)] $\varphi(z)=-\lambda z$ with $|\lambda|=1$ for some $\lambda\in \CC$.
			\item[(ii)] $$
			\varphi(z)={\bar{a}\beta(m+1)^2\beta(m+2)^2  \over a\bar{w}[(m+2)\beta(m+1)^4-(m+1)\beta(m)^2\beta(m+2)^2] }\cdot{ a-z \over 1-\bar{a}z }$$ for some $a\in \DD\backslash{0}$.
			\end{enumerate}
	\end{Theorem}

\begin{proof}
  Since	$D_{m,\psi,\varphi}$ is $J_w $-symmetric on $H^2(\beta)$, Theorem \ref{the1} yields that $$
  \varphi(z)=a_0+{\beta(m+1)^2a_1q(z)  \over (m+1)\bar{w}^{m+1}\beta(m)^2p(z) },$$
  	where $ a_0=\varphi(0)$, $ a_1=\varphi'(0)$, $p(z)$ and $q(z)$ are defined as Theorem \ref{the1}. Since $\varphi$ is an automorphism on $\DD$, then there are $ a\in \DD$ and $\lambda\in \CC$ with $|\lambda|=1$ such that for any $z\in \DD$, $$
  	a_0+{\beta(m+1)^2a_1q(z)  \over (m+1)\bar{w}^{m+1}\beta(m)^2p(z) }=\lambda{  a-z\over 1-\bar{a}z },$$
  	which is equivalent to
  		\begin{equation}\label{equ14}
  		\begin{aligned}
  	 &(m+1)a_0\beta(m)^2\bar{w}^{m+1}p(z)-(m+1)a_0\bar{a}\beta(m)^2\bar{w}^{m+1}zp(z)\\&+a_1\beta(m+1)^2q(z)-a_1\bar{a}\beta(m+1)^2zq(z)\\=&(m+1)\lambda a\beta(m)^2 \bar{w}^{m+1}p(z)-(m+1)\lambda \beta(m)^2 \bar{w}^{m+1}zp(z)
  	 \end{aligned}
   \end{equation}
  	 for any $z\in \DD$.
  	 Considering the constant in (\ref{equ14}), we get
    $$a_0=\lambda a .$$
  Similarly, considering the coefficient of $z$ and $z^2$,   we get
  	 \begin{equation}\label{equ16}
  	 	\begin{aligned}
  	 & {(m+1)\beta(m)^2  \over \beta(m+1)^2 }\bar{w}^{m+2}a_0^2-\bar{a}\bar{w}^{m+1}a_0+\bar{w}^{m+1}a_1\\=&{ (m+1)\beta(m)^2  \over \beta(m+1)^2 }\bar{w}^{m+2}\lambda aa_0-\lambda \bar{w}^{m+1}
  	   \end{aligned}
    \end{equation}
  	  and
  	  \begin{equation}\label{equ15}
  	  \begin{aligned}
  	  &{ (m+1)(m+2)! \beta(m)^2  \over 2\beta(m+2)^2 }\bar{w}^{m+3}a_0^3- {(m+1)(m+1)!\beta(m)^2  \over \beta(m+1)^2 }\bar{w}^{m+2}\bar{a}a_0^2\\+&{(m+2)!\beta(m+1)^2  \over \beta(m+2)^2 }\bar{w}^{m+2}a_0a_1-(m+1)!\bar{w}^{m+1}\bar{a}a_1\\
  	  =&{ (m+1)(m+2)! \beta(m)^2  \over 2\beta(m+2)^2 }\bar{w}^{m+3}\lambda aa_0^2-{(m+1)(m+1)!\beta(m)^2  \over \beta(m+1)^2 }\bar{w}^{m+2}\lambda a_0
  	\end{aligned}
  \end{equation}
for any $w\in \DD$.	

  If $a=0$, then $a_0=\lambda a=0$. Therefore, (\ref{equ16}) deduces that $ a_1=-\lambda$, which implies that $$ p(z)={m!\over \beta(m)^2} ~{\rm and}~  q(z)={ (m+1)!\bar{w}^{m+1}z \over \beta(m+1)^2}.$$ Hence, $\varphi(z)=-\lambda z {\rm ~with~} |\lambda|=1.$

   If $ a\neq 0$, $a_0=\lambda a$ and (\ref{equ15}) give that $$
   a_1={ (m+1)\beta(m)^2\beta(m+2)^2\bar{w} \lambda^2a(|a|^2-1)  \over \beta(m+1)^2[(m+2)\beta(m+1)^2\bar{w}\lambda a-\beta(m+2)^2\bar{a}]  },$$
   which with (\ref{equ16}) yield that $$
    \lambda={\bar{a}\beta(m+1)^2\beta(m+2)^2  \over a\bar{w}[(m+2)\beta(m+1)^4-(m+1)\beta(m)^2\beta(m+2)^2] }.$$
    The proof is complete.
\end{proof}

As an application of Theorem \ref{the1}, we investigate the necessary and sufficient conditions for $J_w$-symmetric operator $ D_{m,\psi,\varphi}$ to be Hermitian  and normal. Recall that a bounded linear operator $T$ is   Hermitian if $T=T^*$.
An operator $ T $ on $\mathcal{H}$ is normal if and only if $TT^*=T^*T$, or for any $x\in \mathcal{H}$, $\|Tx\|=\|T^\ast x\|$.

\begin{Theorem}
	Let $m\in \NN$, $\varphi$ be an analytic self-map of $\DD$  and $\psi\in H(\DD)$ be not identically zero such that $ D_{m,\psi,\varphi}$ is bounded and   $J_w$-symmetric on $H^2(\beta)$. Then $ D_{m,\psi,\varphi}$ is Hermitian if and only if
 $$\psi ^{(m)}(0),\varphi'(0)\in \RR \, \,~~~~~~\mbox{and}\, \,~~~\overline{\varphi(0)}=\bar{w}\varphi(0).$$
	\end{Theorem}

\begin{proof} It is clear that $ D_{m,\psi,\varphi}$ is Hermitian if and only if $$D_{m,\psi,\varphi}K_\alpha(z)=D_{m,\psi,\varphi}^\ast K_\alpha(z) $$ for any $z,\alpha\in \DD $.
	Since $ D_{m,\psi,\varphi}$ is $J_w$-symmetric, then for any $z,\alpha\in \DD $, $$
		J_wD^\ast_{m,\psi,\varphi} K_\alpha(z)=	D_{m,\psi,\varphi}J_wK_\alpha(z).
$$
Therefore, $ D_{m,\psi,\varphi}$ is Hermitian if and only if for any $z,\alpha\in \DD $,
\begin{align}\label{equ17}
J_wD^\ast_{m,\psi,\varphi} K_\alpha(z)=J_wD_{m,\psi,\varphi} K_\alpha(z)=D_{m,\psi,\varphi}J_wK_\alpha(z).
\end{align}
Since
 \begin{align*}
	J_wD_{m,\psi,\varphi} K_\alpha(z)= &J_w{ \psi(z)}K_\alpha^{(m)}(\varphi(z))\\
	=&J_w{ \psi(z)}\sum_{n=m}^\infty{n!  \over(n-m)! }{ \bar{\alpha}^n\varphi(z)^{n-m} \over \beta(n)^2}\\
	=&\overline{ \psi(w\bar{z})}\sum_{n=m}^\infty{n!  \over(n-m)! }{ \alpha^n\overline{\varphi(w\bar{z})}^{n-m} \over  \beta(n)^2}
\end{align*}	
	and
\begin{align*}
	D_{m,\psi,\varphi}J_wK_\alpha(z)=	
	\psi(z)\sum_{n=m}^\infty{n!\over (n-m)!}{\varphi(z)^{n-m}(\bar{w}\alpha)^n \over  \beta(n)^2  }	
\end{align*}
for any $z,\alpha\in \DD$,		
then (\ref{equ17}) is equivalent to
\begin{align}\label{equ18}
\overline{ \psi(w\bar{z})}\sum_{n=m}^\infty{n!  \over(n-m)! }{ \alpha^n\overline{\varphi(w\bar{z})}^{n-m} \over  \beta(n)^2}=\psi(z)\sum_{n=m}^\infty{n!\over (n-m)!}{\varphi(z)^{n-m}(\bar{w}\alpha)^n \over  \beta(n)^2  }	
\end{align}
for any $z,\alpha\in \DD $. Letting $z=0$ in (\ref{equ18}), we have
\begin{align}\label{equ19}
	\overline{\psi(0)}\sum_{n=m}^\infty{n!\over (n-m)!}{ \overline{\varphi(0)}^{n-m}\alpha^n \over  \beta(n)^2}=\psi(0)\sum_{n=m}^\infty{n!\over (n-m)!}{\bar{w}^n\varphi(0)^{n-m} \alpha^n \over \beta(n)^2 }
	\end{align}
for any $\alpha\in \DD$.
Considering the coefficients of $\alpha^m$ and $\alpha^{m+1}$ respectively, we obtain that $$\overline{\psi(0)}=\bar{w}^m\psi(0)$$ and $$
 \overline{\psi(0)\varphi(0)}=\bar{w}^{m+1}\psi(0)\varphi(0),$$
which means that $$ \overline{\varphi(0)}=\bar{w}\varphi(0).$$
Therefore,
\begin{equation}\label{equ23}
	\begin{aligned}
	\overline{p(w\bar{z})}=&\sum_{n=m}^\infty{n!\over (n-m)!}{w^{n-m}\overline{\varphi(0)}^{n-m}(\bar{w}z)^{n-m}  \over  \beta(n)^2 }\\=&\sum_{n=m}^\infty{n!\over (n-m)!}{ w^{n-m}\bar{w}^{n-m}\varphi(0)^{n-m}(\bar{w}z)^{n-m} \over \beta(n)^2 }\\
	=&\sum_{n=1}^\infty{n!\over (n-m)!}{(\varphi(0)\bar{w}z)^{n-m}  \over \beta(n)^2 }\\=&p(z)
\end{aligned}	
\end{equation}
and
\begin{equation}\label{equ24}
	\begin{aligned}
	\overline{q(w\bar{z})}=&	\sum_{n={m+1}}^{\infty}{n!\over (n-m-1)!}{ w^n\overline{\varphi(0)}^{n-m-1}(\bar{w}z)^{n-m}  \over \beta(n)^2 }\\=&\sum_{n={m+1}}^{\infty}{n!\over (n-m-1)!}{ w^{m+n}\bar{w}^{n-m-1}{\varphi(0)}^{n-m-1}\bar{w}^nz^{n-m} \over \beta(n)^2 }\\
	=&\sum_{n={m+1}}^{\infty}{n!\over (n-m-1)!}{ w^{m+n}\bar{w}^{n-m-1}\bar{w}^{m+1}w^{m+1}{\varphi(0)}^{n-m-1}\bar{w}^nz^{n-m} \over \beta(n)^2 }\\=&w^{2m+1}q(z).
\end{aligned}	
\end{equation}
 Differentiating the equation $(\ref{equ18})$ $m$ times with respect to $\alpha$, we get
 \begin{align}\label{equ20}
\overline{ \psi(w\bar{z})}\sum_{n=m}^\infty{(n!)^2\over [(n-m)!]^2}{ \alpha^{n-m}\overline{ \varphi(w\bar{z})}^{n-m} \over \beta(n)^2  }=\psi(z)\sum_{n=m}^\infty{(n!)^2\over [(n-m)!]^2}{ \bar{w}^{n}\alpha^{n-m} \varphi(z)^{n-m}\over \beta(n)^2  }
\end{align}
for any $\alpha, z\in \DD$. Letting $\alpha=0$ in (\ref{equ20}), we have that $$\overline{ \psi(w\bar{z})}=\bar{w}^m\psi(z) $$ for any $z\in \DD$. Since  $ D_{m,\psi,\varphi}$ is $J_w$-symmetric, Theorem \ref{the1} yields that
\begin{align*}
&{\overline{\psi^{(m)}(0)}\beta(m)^2\over (m!)^2 }\sum_{n=m}^\infty{n!\over (n-m)!}{ w^{n-m}\overline{\varphi(0)}^{n-m}(\bar{w}z)^n \over \beta(n)^2 }\\
=&{\overline{\psi^{(m)}(0)}\beta(m)^2\over (m!)^2 }\sum_{n=m}^\infty{n!\over (n-m)!}{ \bar{w}^{n}\varphi(0)^{n-m}z^n \over \beta(n)^2 }\\
=&{{\psi^{(m)}(0)}\beta(m)^2\over (m!)^2 }\sum_{n=m}^\infty{n!\over (n-m)!}{ \bar{w}^{n}\varphi(0)^{n-m}z^n \over \beta(n)^2 },
\end{align*}
which implies
   $$\overline{\psi^{(m)}(0)}= \psi^{(m)}(0).$$
Therefore, by (\ref{equ20})
 \begin{align}\label{equ21}
 \bar{w}^m\sum_{n=m}^\infty{[n!]^2\over [(n-m)!]^2}{ \alpha^{n-m}\overline{ \varphi(w\bar{z})}^{n-m} \over \beta(n)^2  }=\sum_{n=m}^\infty{[n!]^2\over [(n-m)!]^2}{ \bar{w}^{n}\alpha^{n-m} \varphi(z)^{n-m}\over \beta(n)^2  }
 \end{align}
for any $\alpha, z\in \DD$. Considering the coefficient of $\alpha$ in (\ref{equ21}), we have
 $$ \overline{\varphi(w\bar{z})}=\bar{w}\varphi(z)$$ for any $z\in \DD$, which equivalent to,
\begin{equation}\label{equ27}
	\begin{aligned}
	\overline{\varphi(0)}+{\beta(m+1)^2\overline{\varphi'(0)}  \over (m+1)w^{m+1}\beta(m)^2 }{ \overline{q(w\bar{z})} \over\overline{p(w\bar{z})}  }=&\bar{w}\varphi(0)+{\beta(m+1)^2\overline{\varphi'(0)}  \over (m+1)w^{m+1}\beta(m)^2 }{ w^{2m+1}q(z) \over p(z) }\\
	=&\bar{w}\varphi(0)+{ \beta(m+1)^2 \overline{\varphi'(0)}w^m\bar{w}^{m+1} \over (m+1)\bar{w}^{m+1}\beta(m)^2 }{q(z)\over p(z)}\\=&\bar{w}\left( \varphi(0)+{ \beta(m+1)^2 \overline{\varphi'(0)} \over (m+1)\bar{w}^{m+1}\beta(m)^2 }{q(z)\over p(z)}\right)\\
	=&\bar{w}\left( \varphi(0)+{ \beta(m+1)^2 {\varphi'(0)} \over (m+1)\bar{w}^{m+1}\beta(m)^2 }{q(z)\over p(z)}\right),
\end{aligned}
\end{equation}
which implies that $\varphi'(0) \in \RR$.

Conversely, assume that $\psi ^{(m)}(0),\varphi'(0)\in \RR$ and $ \overline{\varphi(0)}=\bar{w}\varphi(0)$. Obviously, it is sufficient to verify that equation (\ref{equ18}) holds.
Since $ \overline{\varphi(0)}=\bar{w}\varphi(0)$ and
\begin{align*}
\overline{\psi(w\bar{z})}=&{\beta(m)^2\overline{\psi^{(m)}(0)}\over (m!)^2}\sum_{n=m}^\infty{n!\over (n-m)!}{w^{n-m}\overline{\varphi(0)}^{n-m}(\bar{w}z)^n  \over  \beta(n)^2 }\\=&{\beta(m)^2{\psi^{(m)}(0)}\over (m!)^2}\sum_{n=m}^\infty{n!\over (n-m)!}{{\varphi(0)}^{n-m}(\bar{w}z)^n  \over  \beta(n)^2 }\\=&\bar{w}^m\psi(z),
\end{align*}
 we see that (\ref{equ23}) and  (\ref{equ24}) hold. Thus  from (\ref{equ27}), we obtain that
\begin{align*}
	\overline{\varphi(w\bar{z})}=\bar{w}\varphi(z).
	\end{align*}
Therefore,
\begin{align*}
\overline{ \psi(w\bar{z})}\sum_{n=m}^\infty{n!  \over(n-m)! }{ \alpha^n\overline{\varphi(w\bar{z})}^{n-m} \over  \beta(n)^2}=&	\bar{w}^m\psi(z)\sum_{n=m}^\infty{n!  \over(n-m)! }{ \alpha^n\bar{w}^{n-m}\varphi(z)^{n-m} \over \beta(n)^2 }\\
=&\psi(z)\sum_{n=m}^\infty{n!  \over(n-m)! }{ \alpha^n\bar{w}^{n}\varphi(z)^{n-m} \over \beta(n)^2 }
\end{align*}
for any $z,\alpha\in \DD$. The proof is complete.	
\end{proof}

If $\varphi(0)=0$, the following result implies that every $J_w$-symmetric  operator $ D_{m,\psi,\varphi}$ is normal.

\begin{Theorem}
	Let $m\in \NN$, $\varphi$ be an analytic self-map of $\DD$ with $\varphi(0)=0$  and $\psi\in H(\DD)$ be not identically zero such that $ D_{m,\psi,\varphi}$ is bounded and $J_w$-symmetric  on $H^2(\beta)$. Then $ D_{m,\psi,\varphi}$ is normal.
\end{Theorem}

\begin{proof}
	 Obviously, $\varphi(0)=0$ gives  $$p(z)={m!\over \beta(m)^2} ~and~ q(z)={(m+1)!\bar{w}^{m+1}z \over \beta(m+1)^2 }.$$ Since $ D_{m,\psi,\varphi}$ is $J_w$-symmetric, Theorem \ref{the1} yields that  $$\varphi(z)=a_1z ~{\rm and}~ \psi(z)={a_2\over m!}z^m,$$ where $ a_1=\varphi'(0)$ and $a_2=\psi^{(m)}(0).$  Then for $j\in \NN^+$,
	 \begin{align*}
	 	\|D_{m,\psi,\varphi}e_j\|^2=&\sum_{n=0}^\infty|\langle D_{m,\psi,\varphi}e_j, e_n \rangle |^2\\
	 	=&\sum_{n=0}^\infty|\langle\psi e_j^{(m)}\circ\varphi , {z^n\over \beta(n)}\rangle |^2\\
	 	=&\sum_{n=0}^\infty|\langle  { j!a_2a_1^{j-m}z^j \over m!(j-m)! \beta(j)}, {z^n\over \beta(n)}\rangle |^2
	 	\end{align*}
 	and
 	\begin{align*}
 			\|D^\ast_{m,\psi,\varphi}e_j\|^2=&\sum_{n=0}^\infty|\langle  D^\ast_{m,\psi,\varphi}e_j, e_n\rangle |^2
 			=\sum_{n=0}^\infty|\langle e_j,D_{m,\psi,\varphi}e_n\rangle |^2\\
 		=&\sum_{n=0}^\infty|\langle {z^j\over \beta(j)},\psi e_n^{(m)}\circ\varphi \rangle |^2\\
 		=&\sum_{n=0}^\infty|\langle {z^j\over \beta(j)}, { n!a_2a_1^{n-m}z^n \over m!(n-m)! \beta(n)}\rangle |^2.
 	\end{align*}
 Therefore, for $j\in \NN^+$,
 $$
 \|D_{m,\psi,\varphi}e_j\|^2=	\|D^\ast_{m,\psi,\varphi}e_j\|^2=|a_2a_1^{j-m}|^2\left({j!\over m!(j-m)!}  \right)^2.$$
Hence $ D_{m,\psi,\varphi}$ is normal. The proof is complete.
\end{proof}

The following result finds a necessary and sufficient condition for a $J_w$-symmetric   operator $ D_{m,\psi,\varphi}$ to be normal.

\begin{Theorem}
	Let $m\in \NN$, $\varphi$ be an analytic self-map of $\DD$ with $\varphi'(0)=0$  and $\psi\in H(\DD)$ be not identically zero such that $ D_{m,\psi,\varphi}$ is bounded and  $J_w$-symmetric on $H^2(\beta)$. Then $ D_{m,\psi,\varphi}$ is normal if and only if $ \overline{\varphi(0)}=\bar{w}\varphi(0)$.
\end{Theorem}

\begin{proof}
	Since $ D_{m,\psi,\varphi}$ is $J_w$-symmetric and $\varphi'(0)=0$, Theorem \ref{the1} deduces that
	 $$\varphi(z)=\varphi(0) {\rm ~and~}   	\psi(z)={ \psi^{(m)}(0)\beta(m)^2 \over (m!)^2 }K^{[m]}_{w\overline{\varphi(0)}}(z).$$ Since for any $f\in H^2(\beta)$, we have
	 \begin{align*}
	 	\left \langle  f,D_{m,\psi,\varphi}^\ast K^{[m]}_{\alpha}\right\rangle =&	\left \langle D_{m,\psi,\varphi}f, K^{[m]}_{\alpha}\right\rangle \\
	 	=&\psi^{(m)}(\alpha)f^{(m)}(\varphi(\alpha)) \\
	 	=&	\left \langle f, \overline{\psi^{(m)}(\alpha)}K^{[m]}_{\varphi(\alpha)}\right\rangle.
	 	\end{align*}
   Then
      $$
 	D_{m,\psi,\varphi}^\ast K^{[m]}_{\alpha}=\overline{\psi^{(m)}(\alpha)}K^{[m]}_{\varphi(\alpha)}$$ for any $\alpha\in \DD$. Hence,
	 for any $\alpha,z\in \DD$,  
	\begin{align*}
	&	D_{m,\psi,\varphi}^\ast D_{m,\psi,\varphi}K_\alpha(z)\\=&D_{m,\psi,\varphi}^\ast\psi(z)K^{(m)}_\alpha(\varphi(z))\\
		=&{ \psi^{(m)}(0)\beta(m)^2 \over (m!)^2 }\sum_{n=m}^\infty { n! \over (n-m)! }{ \bar{\alpha}^n \varphi(0)^{n-m}\over  \beta(n)^2}
		D_{m,\psi,\varphi}^\ast\psi(z)K^{[m]}_{w\overline{\varphi(0)}}(z)\\
		=&{ \psi^{(m)}(0)\beta(m)^2 \over (m!)^2 }\sum_{n=m}^\infty { n! \over (n-m)! }{ \bar{\alpha}^n \varphi(0)^{n-m}\over  \beta(n)^2}\overline{\psi^{(m)}(w\overline{\varphi(0)})}K^{[m]}_{{\varphi(0)}}(z)\\
		=&{ |\psi^{(m)}(0)|^2\beta(m)^4 \over (m!)^4 }\sum_{n=m}^\infty { n! \over (n-m)! }{ \bar{\alpha}^n \varphi(0)^{n-m}\over  \beta(n)^2}\\
		&\,\,\,\,\,\,\,\,\,\,\cdot\sum_{n=m}^\infty { (n!)^2 \over [(n-m)!]^2 }{ |\varphi(0)|^{2(n-m)} \over\beta(n)^2 }\sum_{n=m}^\infty{ n! \over (n-m)! }{ \overline{\varphi(0)}^{n-m}z^n \over \beta(n)^2}
		\end{align*}
	and
	\begin{align*}
		&D_{m,\psi,\varphi}D_{m,\psi,\varphi}^\ast K_\alpha(z)=	D_{m,\psi,\varphi}\overline{\psi(\alpha)}K^{[m]}_{{\varphi(\alpha)}}(z)\\
		=&\overline{\psi(\alpha)}\psi(z)\left( K^{[m]}_{{\varphi(\alpha)}}(z)\right)^{(m)}\circ\varphi(z)\\
		=&{ |\psi^{(m)}(0)|^2\beta(m)^4 \over (m!)^4 }\overline{K^{[m]}_{w\overline{\varphi(0)}}(\alpha)}K^{[m]}_{w\overline{\varphi(0)}}(z)\left( K^{[m]}_{{\varphi(\alpha)}}(z)\right)^{(m)}\circ\varphi(z)
		\end{align*}
		\begin{align*}
		=&{ |\psi^{(m)}(0)|^2\beta(m)^4 \over (m!)^4 }\sum_{n=m}^\infty { n! \over (n-m)! }{ \bar{w}^{n-m} \varphi(0)^{n-m}z^n\over  \beta(n)^2}\\
		&\,\,\,\,\,\,\,\,\,\,\cdot\sum_{n=m}^\infty{ n! \over (n-m)! }{ w^{n-m}\overline{\varphi(0)}^{n-m}\bar{a}^n \over \beta(n)^2}\sum_{n=m}^\infty { (n!)^2 \over [(n-m)!]^2 }{ |\varphi(0)|^{2(n-m)} \over\beta(n)^2 }.
\end{align*}
Therefore, $D_{m,\psi,\varphi}$ is normal if and only if
\begin{equation}\label{equ25}
	\begin{aligned}
	&\sum_{n=m}^\infty { n! \over (n-m)! }{ \bar{\alpha}^n \varphi(0)^{n-m}\over  \beta(n)^2}\sum_{n=m}^\infty{ n! \over (n-m)! }{ \overline{\varphi(0)}^{n-m}z^n \over \beta(n)^2}\\
	=&\sum_{n=m}^\infty { n! \over (n-m)! }{ \bar{w}^{n-m} \varphi(0)^{n-m}z^n\over  \beta(n)^2}\sum_{n=m}^\infty{ n! \over (n-m)! }{ w^{n-m}\overline{\varphi(0)}^{n-m}\bar{a}^n \over \beta(n)^2}
	\end{aligned}
\end{equation}
for any $z,\alpha\in \DD$. Considering the coefficient of $\bar{\alpha}^mz^{m+1}$ in (\ref{equ25}), we have $$\overline{\varphi(0)}=\bar{w}\varphi(0).$$

Conversely, assume that $\overline{\varphi(0)}=\bar{w}\varphi(0)$. By a simple calculation, equation (\ref{equ25}) holds. The proof is complete.
\end{proof}

\end{document}